\documentclass[11pt,reqno]{amsart}

\usepackage{mathtools}
\usepackage{amssymb}
\usepackage{hyperref}
\usepackage{imakeidx}
\usepackage{tikz}
\usepackage{manfnt}
\usetikzlibrary{%
  matrix,%
  calc,%
  arrows%
}
\usepackage{graphicx}
\usepackage{fullpage}

\newcommand{\Q}{\mathbb{Q}}

\newcommand{\Z}{\mathbb{Z}}
\newcommand{\comment}[1]{}

\theoremstyle{definition}
\newtheorem{theorem}{Main Theorem}[section]

\newtheorem{corollary}[theorem]{Corollary}

\newtheorem{lemma}[theorem]{Lemma}
\newtheorem{proposition}[theorem]{Proposition}

\theoremstyle{definition}
\newtheorem{definition}[theorem]{Definition}

\newtheorem{example}[theorem]{Example}

\numberwithin{equation}{subsection}
\theoremstyle{plain}

\theoremstyle{definition}
\newtheorem{defn}[theorem]{Definition}

\newtheorem*{theorem*}{Theorem}
\newtheorem*{lem*}{Lemma}
\newtheorem*{con*}{Conjecture}
\newtheorem{Claim*}{Claim}
\newtheorem*{defn*}{Definition}

\newtheorem*{rem*}{Remark}

\makeatletter
\def\imod#1{\allowbreak\mkern10mu\left({\operator@font mod}\,\,#1\right)}
\makeatother

\begin{document}

\bibliographystyle{plain}

\title[Root Counts Modulo P Squared 110117]{Counting Roots of Polynomials Over $\mathbb{Z}/p^2\mathbb{Z}$}
\author{Trajan Hammonds, Jeremy Johnson, Angela Patini, and Robert M. Walker}

\address{Department of Mathematics, Carnegie Mellon University, Pittsburgh, PA, 15289}
\email{thammond@andrew.cmu.edu}

\address{Department of Mathematics, Humboldt State University, Arcata, CA, 95521}
\email{jsj132@humboldt.edu}

\address{Department of Mathematics, University of Pennsylvania, Philadelphia, PA, 19104}
\email{apatini@sas.upenn.edu}

\address{Department of Mathematics, University of Michigan, Ann Arbor, MI, 48109}
\email{robmarsw@umich.edu}

\parskip=10pt plus 2pt minus 2pt

\begin{abstract} Until recently, the only known method of finding the roots of polynomials over prime power rings, other than fields, was brute force. One reason for this is the lack of a division algorithm,  obstructing the use of greatest common divisors. Fix a prime $p \in \Z$  and $f \in (\mathbb{Z}/p^n\mathbb{Z})[x]$ any nonzero polynomial of degree $d$ whose coefficients are not all divisible by $p$. For the case $n=2$, we prove a new efficient algorithm to count the roots of $f$ in $\mathbb{Z}/p^2\mathbb{Z}$ within time polynomial in $(d+\operatorname{size}(f)+\log{p})$, and record a concise formula for the number of roots, formulated by  
Cheng, Gao, Rojas, and Wan. 
 \end{abstract}


\maketitle


\section{Introduction}

\indent Since the days of Diophantus, mathematicians have been interested in finding rational or integer solutions to polynomial equations. In the 1940s, Andr{\'e} Weil proved the Riemann hypothesis for zeta-functions of nonsingular curves over finite fields \cite{weil1948courbes}.  
In 1949, Weil proposed enticing conjectures that connect finding solutions to polynomials over finite fields with studying the geometry of complex algebraic varieties \cite{Weil49}. Weil proved these conjectures in the case of curves, yielding a bound for 
counting the number of points on a curve over a finite field--the  \textbf{Hasse-Weil bound}: \begin{equation*}
|N_q - (q+1)| \leq 2g\sqrt{q},
\end{equation*} where $q$ is a prime power and $N_q$ is the number of points over $\mathbb{F}_q^2$ on a curve with genus $g$. 
Such bounds on point counts extend to higher dimensions, per work of Weil, Deligne, Dwork, and others.

We wish to count roots over the prime power ring $\mathbb{Z}/p^k\mathbb{Z}$. 
That said, the usual approaches do not work since 
the polynomial ring $(\mathbb{Z}/p^k\mathbb{Z})[x]$ does not have unique factorization when $k \ge 2$.  
Thus we must sleuth for alternate approaches to count roots of nonconstant univariate polynomials over $\mathbb{Z}/p^k \Z$, since traditional methods for factoring and root counting over finite fields are unavailable.

As a backdrop, suppose $p \in \Z$ is a prime, both $m, v \in \mathbb{Z}_+$, $f \in \mathbb{Z}[x_1, \dots x_v]$ is a nonzero polynomial with at least one coefficient being a unit modulo $p$, and $N_m(f)$ denotes the number of solutions to $f \equiv 0 \mod p^m$ in the ring $(\mathbb{Z}/p^m \Z )^v$. 
Consider the \textbf{Igusa Poincar{\'e} Series} \cite{Igusa00}: 
\begin{equation*}
Q(f;t) = \sum_{m>0} N_m(f) \cdot t^m \in \Z [[ t ]].
\end{equation*}
Igusa's proof that $Q(f; t)$ is rational \cite{Igusa78}, solving a conjecture of Borevich and Shafarevich, relied on Hironaka's resolution of singularities \cite{10.2307/1970486}, which runs in exponential time. Zuniga-Galindo \cite{zuniga2003computing} later derived an algorithm to compute $Q(f;t)$, where the dependence on $v$ in the complexity was of order $8$. While one could in principle use standard generating function tricks to then extract $N_m(f)$ for any given $m$, Zuniga-Galindo's algorithm only works in the case where $f$ splits completely into linear factors over $\Q$ -- a severe restriction. Cheng, Gao, Rojas, and Wan, during  a meeting at the American Institute for Mathematics (AIM) in May 2017, found an explicit formula for $N_2(f)$ when $v=1$, but without a proof or complexity bound. We prove their formula is correct and that it has near-quadratic complexity. 

Going forward, given a prime $p \in \Z_+$, and $k \in \Z_{+}$, we view the set $\Z/ p^k \Z: = \{\overline{0} , \overline{1}, \ldots, \overline{p^k - 1} \}$ as a ring, and let $\pi_{p^k} \colon \Z [x] \twoheadrightarrow (\Z / p^k \Z)[x]$ denote  the surjective ring homomorphism defined by  
$$\pi_{p^k} \left( \sum_{i=0}^e c_i x^{e-i} \right) := \sum_{i=0}^e \overline{c_i} \cdot  x^{e-i},$$ where $\overline{c} : = \pi_{p^k} (c) \in \Z / p^k  \Z$ when $c \in \Z$  \cite[Ch.~9]{DummitFoote04}. 
Given a polynomial $g \in (\Z / p^k \Z )[x]$, we let $\widetilde{g} \in \Z [x]$ denote the lift of $g$--read, $\pi_{p^k} (\widetilde{g}) = g$--whose coefficients all lie between 0 and $p^k-1$. Also, for $g \in (\Z / p \Z) [x]$, we say a root of multiplicity one is simple, and a root is degenerate otherwise. 

\begin{defn}\label{defn:reduction-mod-prime-powers}
For any nonconstant polynomial $f \in \mathbb{Z}[x]$, any prime $p \in \Z_+$, and any $k \in \Z_+$, let $V_{p^k}(f) = \{\zeta \in \Z / p^k \Z \colon [\pi_{p^k}(f)] (\zeta) =  0 \in \Z / p^k \Z \}$. 
 Also, we set $A_k(p) : = \{0 , 1, \ldots, p^k - 1\} \subseteq \Z$.  
\end{defn}  


\begin{definition}\label{hypo:ascending-seperable-factorization01}
Let $f \in \mathbb{Z}[x] \setminus  \{0\}$ be a nonconstant polynomial of degree $d$. Fix any prime $p$ not dividing every coefficient of $f$.  We define the maximal multiplicity $\ell$ of a root of $f$ modulo $p$, 
and a series of polynomials $f_1, \ldots, f_\ell, g , h_1, h_2, t$ all in $(\mathbb{Z}/p\mathbb{Z})[x]$ and polynomials $\mathcal{L}_1, \ldots, \mathcal{L}_\ell \in \Z[x]$. 
\begin{enumerate}
\item 
We can factor $h_1 : = \pi_ p (f)$ as 
\begin{equation}\label{eqn:ascending-separable-factorization}
h_1 = \pi_p (f) = f_1 f_2^2 \cdots f_\ell^\ell g \in (\mathbb{Z}/p\mathbb{Z})[x],
\end{equation}
where 
\begin{enumerate}
\item $\ell$ is the maximal multiplicity of a root $r \in \mathbb{Z}/p\mathbb{Z}$ of $h_1$--if $h_1$ has any;  
\item the $f_i \in (\Z / p \Z)[x]$ are \textit{monic, separable}, and \textit{pairwise coprime}; and 
\item $g \in (\Z / p \Z)[x]$ has no roots in $\Z  / p \Z$. 
\end{enumerate}
Thus the degree of $f_i$ equals the number of distinct roots of $h_1$ in $\Z /p  \Z$ of multiplicity $i$. The reader can consult \cite[Ch.~8,9,13]{DummitFoote04} for relevant background definitions in the setting of univariate polynomial rings over a field. 
\item Suppose that $f_i = \prod_{j=1}^{\deg(f_i)} L_{i , j}$ as a product (possibly empty) of distinct linear terms in $(\Z/p\Z)[x]$. We define $\mathcal{L}_i \in \Z[x]$ to be $\mathcal{L}_i = \prod_{j=1}^{\deg(f_i)} \widetilde{L_{i , j}}$. Note that $\pi_p (\mathcal{L}_i) = f_i$.  
\item We also define polynomials $t, h_2 \in (\mathbb{Z}/p\mathbb{Z})[x]$ via $$t := \pi_p \Big[ \frac{1}{p} \left(f - \widetilde{g} \cdot \prod_{i=1}^\ell \mathcal{L}_i^i  \right) \Big], \quad h_2:=\mathrm{gcd}(f_2 \cdots f_\ell, t).$$ 
\end{enumerate} 
\end{definition}

\begin{definition}\label{defn:size-of-integer-poly}
Fix a degree-$d$ polynomial $f \in  \Z[x] \setminus \{0\}$ written as $f(x) = c_0 + c_1  x + ... + c_d x^{d}$. In terms of the natural logarithm, we define the 
\textbf{computational size} of $f$ to be $$\operatorname{size}(f) = \sum^d_{i=0} \log (2+|c_i|);$$ it is simply the number of bits needed to record the above monomial term expansion of $f$. 
\end{definition} 

We now state the main result of this note. 

\begin{theorem}\label{thm:deterministic-root-counting01}
\textit{With notation as in Definitions \ref{defn:reduction-mod-prime-powers}, \ref{hypo:ascending-seperable-factorization01}, and \ref{defn:size-of-integer-poly}, 
\begin{enumerate}
\item We have 
\begin{equation}\label{eqn:mod-prime-square-root-count}
\# V_{p^2}(f) = \# \left\{a \in A_2(p) \colon f (a) \equiv 0 \mod p^2 \right\} = \mathrm{deg}(f_1) + p \cdot \mathrm{deg}(h_2),
\end{equation}
where $\mathrm{deg}$ stands for polynomial degree. 
\item  The polynomials 
$t$, $f_1$, and $h_2$ can be computed deterministically in time that is polynomial in  $d + \operatorname{size}(f) + \mathrm{log}(p),$ 
 where $d = \deg(f)$, counting the necessary arithmetic operations. 
\end{enumerate}
}
\end{theorem}
\noindent While the first term in formula \eqref{eqn:mod-prime-square-root-count} counts the roots modulo $p^2$ that descend to simple roots modulo $p$, the second term counts the roots modulo $p^2$ that descend to degenerate roots modulo $p$.

\section{Preliminaries for the Proof}
Throughout, $p$ is an arbitrary prime number. 
 We state a proposition together with two versions of Hensel's Lemma, a crucial tool for proving Theorem \ref{thm:deterministic-root-counting01}(1).

\begin{proposition}[Cf., {\cite[Sec.~13.5, Prop.~33]{DummitFoote04}}]\label{prop:simple-roots-mod-p}
If $g \in (\mathbb{Z}/p\mathbb{Z}) [x]$ is nonconstant,  and there is an $r \in \mathbb{Z}/p\mathbb{Z}$ such that $(x-r) \mid g$ but $(x-r)^2 \nmid g$, then $g'(r) \neq 0  $ in $\Z  / p \Z$. 
\end{proposition}


\noindent In \cite[Sec. 2.6, Thm.~2.23 + paragraph between Examples 11-12]{montgomery1991introduction}, a derivation of both versions of Hensel's Lemma below is given via Taylor expansion. 
We note that \cite[Sec.~2.6]{montgomery1991introduction} phrases both versions of Hensel's Lemma in terms of an arbitrary integer $r$ rather than stipulating $0 \le r \le p-1$.   

\begin{lemma}[Hensel's Lemma Version I]\label{lem:hensel01}
Let \(f \in  \Z [x]\) be nonconstant,  and suppose there is an \(r \in A_1 (p)\) with \([\pi_p (f)] (\overline{r}) = 0\). If \([\pi_p (f)]'(\overline{r}) \neq 0\), then there exists an $s \in A_2 (p)$ such that $[\pi_{p^2}(f)] (\overline{s}) = 0$ in $\Z / p^2 \Z$ and $s \equiv r \mod p$, namely, $s = \widetilde{t}$ where \(t := \overline{r} - \left(\overline{f'(r)}\right)^{-1} \cdot \overline{f(r)}   \in \Z / p^2 \Z \). 
Moreover, $s$ is unique.  
\end{lemma}


\begin{lemma}[Hensel's Lemma Version II]\label{lem:hensel02} 
Let \(f  \in \Z [x]\) be nonconstant, and suppose there exists \(r\) in \(A_1 (p)\) such that \(f(r) \equiv 0 \mod p^k\), where $k \in \Z_{+}$. If \(f'(r) \equiv 0 \mod p\), then 
\begin{equation*}
s \equiv r \mod p^k \implies f(s) \equiv f(r) \mod p^{k+1}. 
\end{equation*}
That is, $f(r+tp^k) \equiv f(r) \mod p^{k+1}$ for all $0 \le t \le p-1$, indeed for all $t \in \mathbb{Z}$.  
\end{lemma}

\noindent Notably, we have $p$ roots $\mathrm{mod} \ p^{k+1}$ when $f(r) \equiv 0 \mod p^{k+1}$. Thus Lemma \ref{lem:hensel02}  can lift roots modulo $p^k$ to roots modulo $p^{k+1}$. 
Conversely, all the roots modulo $p^{k+1}$ are obtained this way.

\section{Proof of the Main Theorem}

\begin{proof}[Proof of Theorem \ref{thm:deterministic-root-counting01}$(1)$] Recall that in Definition \ref{hypo:ascending-seperable-factorization01} we defined polynomials $\mathcal{L}_i \in \Z[x]$ such that $\pi_p (\mathcal{L}_i) = f_i$. Let $U := \{\widetilde{\zeta} \in A_2(p) \colon \zeta \in V_{p^2} (f)\}$, which is the disjoint union of the two sets $$S := \{u \in  U \colon \overline{u} \in V_{p}(\mathcal{L}_1)\} , \mbox{ and } T := U  \setminus S .$$ Recall that we defined $h_2 = \gcd (f_2 \cdots f_\ell, t) \in (\Z / p \Z) [x]$; this monic polynomial  is a product (possibly empty) of distinct linear terms. Let $D(x) \in \Z[x]$ be the lift of $h_2$ constructed analogously to the $\mathcal{L}_i$, taking the corresponding product of the $\widetilde{\bullet}$ lifts of the linear factors.  
To get \ref{thm:deterministic-root-counting01}(1), it suffices to show that as maps of sets (a) $\pi_p |_S:S \rightarrow V_p(\mathcal{L}_1)$ is a bijection, and (b) $\pi_p |_T: T \rightarrow V_p(D)$ is a $p$-to-$1$ surjection. 
But first, we record a lemma. 

\begin{lemma}\label{lem:root-reduction01}
Let $\rho \colon A_2(p) \to A_1(p) $ be the map of sets sending an element $a \in A_2(p)$ to its remainder after long division by $p$. Fix $r \in A_2(p)$. If $f(r) \equiv 0 \mod p^2$, then $f( \rho(r) ) \equiv 0 \mod p$. Equivalently, if $\overline{r} \in V_{p^2} (f)$, then $\overline{\rho(r)} \in V_{p} (f) $ in terms of the bar notation preceding Definition \ref{defn:reduction-mod-prime-powers}. 
\end{lemma}

\noindent Indeed, if $f(a) = \sum_{i=0}^d c_{d-i} a^{d-i}$ for any $a \in \Z$, then $f(r) \equiv \sum_{i=0}^d c_{d-i} ( \rho (r) )^{d-i} = f (\rho(r))$ mod $p$.



\noindent (a) \ $\pi_p |_S$ \textbf{is a bijection:} This is vacuous if $S$ is empty, so we may assume $S$ is non-empty. First, given any element $r \in U$, Lemma \ref{lem:root-reduction01} says $\pi_p (r) = \pi_p(\rho(r)) \in V_p(f)$, meeting the first hypothesis of Hensel's Lemma \ref{lem:hensel01}. Because of our stipulations in defining the polynomials $f_i$ in \eqref{eqn:ascending-separable-factorization}, Proposition \ref{prop:simple-roots-mod-p} applied to $h_1 = \pi_p(f)$ implies that $\pi_p (\rho(r))$  satisfies the second hypothesis under Hensel's Lemma \ref{lem:hensel01} if and only if $\pi_p (\rho(r)) \in V_p(\mathcal{L}_1)$. Equivalently, $r \in S$ and it will be the \textit{unique} lift to $A_2(p)$ of $\rho(r) \in A_1(p)$ as stipulated in Hensel's Lemma \ref{lem:hensel01}, since $r \equiv \rho(r) \mod{p}$. Thus we may conclude that $\pi_p |_S$ is both surjective and injective, hence bijective. 

\noindent Before proceeding, we record another lemma. 

\begin{lemma}\label{lem:degenerate-roots-h1}
Given $r \in A_1(p) $ and $\overline{r} := \pi_p(r) \in \Z /p \Z$, the following assertions are equivalent to saying $ (x- \overline{r})^2 \mid  h_1$:
\begin{enumerate}
\item $\overline{r}$ is a degenerate root of $h_1$, i.e., both $f (r) \equiv 0$ mod $p$ and $f'(r) \equiv 0$ mod $p$. 
\item $(x-\overline{r}) \mid  f_i$ for some unique $i \ge 2$.  
\item $(x-\overline{r}) \mid  f_2 \cdot \ldots \cdot f_\ell$.  
\end{enumerate}
\end{lemma}

\noindent Indeed, per the stipulations on the $f_i$ in  \eqref{eqn:ascending-separable-factorization}, 
all of these assertions mean $f_1 (\overline{r}) \neq 0$ in $\Z / p \Z$. 

\noindent (b) \ $\pi_p |_T$ \textbf{is a $p$-to-$1$ surjection:} This is vacuous if $T$ is empty, so we may assume $T$ is non-empty. We note that $\pi_p (T) \subseteq V_p(\mathcal{L}_2 \cdots \mathcal{L}_\ell)$: given $r \in T$, Lemmas \ref{lem:root-reduction01} and \ref{lem:degenerate-roots-h1} apply to $\rho (r)$. 
Let $$E(x) := f(x) -  \widetilde{g} (x) \cdot   \prod_{i=1}^\ell \mathcal{L}_i^i (x) \in p \Z[x] = \ker \pi_p.$$ Then the integer polynomial $(1/p) \cdot E(x)$ is a lift of $t(x)$.  
Next, since $h_2$ divides $h_1$ in $(\Z/p\Z)[x]$, we note that any $r \in A_1 (p)$ for which $D (r) \equiv 0 \mod p$ also satisfies $\mathcal{L}_i (r) \equiv 0 \mod p$ for some $i \ge 2$ by Lemma \ref{lem:degenerate-roots-h1}. Thus $f(r) \equiv E(r) \mod p^2$. Additionally, $t(\overline{r}) = 0 \in \Z /p \Z$, so $(1/p) E(r) \equiv 0 \mod{p}$, hence $E(r) \equiv 0 \mod p^2$.  Then $f(r) \equiv 0 \mod p^2$, so  
Hensel's Lemma \ref{lem:hensel02} says that  $r$ can be lifted to $p$ \textit{distinct} roots  $s_j = r + j \cdot p \in T$ of $f$ modulo $p^2$ where $0 \le j \le p-1$. 
Thus $V_p(D) \subseteq \pi_p (T)$. 

To conclude that $\pi_p |_T$ is a $p$-to-1 surjection onto $V_p(D)$, it remains to show that conversely, given $u \in T$,  $\overline{u} := \pi_p (u) \in V_p(D)$. 
Since $\pi_p (T) \subseteq V_p(\mathcal{L}_2 \cdots \mathcal{L}_\ell)$, we have $(f_2 \cdots f_\ell) (\overline{u}) = 0  \in \Z/p \Z$ and  $(\mathcal{L}_i (u))^i \equiv 0 \mod p^2$ for some $i \geq 2$. Thus  $f(u) \equiv E(u) \equiv 0 \mod p^2$:  indeed, since $u \in T$, $f(u) \equiv 0 \mod p^2$. It follows that $(1/p) E (u) \equiv 0 \mod p$. Equivalently, $t (\overline{u}) = 0 \in \Z / p \Z$.  
We may conclude that $(x-\overline{u}) | (f_2 \cdots f_\ell)$ and $(x-\overline{u}) | t$ in $(\mathbb{Z}/p\mathbb{Z})[x]$, so by the definition of greatest common divisor $(x-\overline{u}) | h_2$ in $(\mathbb{Z}/p\mathbb{Z})[x]$. Thus $\overline{u} \in V_p(D)$.  This completes the proof of claim (b), so we are done. 
\end{proof}


\begin{corollary}\label{cor:mod-p-roots-without-lifts}
\textit{With notation as in Definitions \ref{defn:reduction-mod-prime-powers} and \ref{hypo:ascending-seperable-factorization01}, exactly  
\begin{equation}\label{eqn:mod-p-roots-without-lifts}
\# \{a \in A_1(p) \colon \overline{a} \in V_p (\mathcal{L}_2 \cdots \mathcal{L}_\ell), \mbox{ } f(a) \not\equiv 0 \mod p^2\} = \deg (f_2 \cdots f_\ell) - \deg (h_2)
\end{equation}
degenerate roots of $f$ modulo $p$ fail to lift to roots of $f$ modulo $p^2$.} 
\end{corollary}

\begin{proof}
To start, continuing from the proof of Theorem \ref{thm:deterministic-root-counting01}(1), the right-hand side is equal to $\# V_p(\mathcal{L}_2 \cdots \mathcal{L}_\ell) - \#V_p (D)$, since $f_2 \cdots f_\ell$ and $h_2$ are separable. 
Our argument for claim (b) in the proof of Theorem \ref{thm:deterministic-root-counting01}(1) suffices to show that $\pi_p (T) = V_p(D) = V_p (\mathcal{L}_2 \cdots \mathcal{L}_\ell) \cap  V_p [(1/p) E]$, and that the set stated in the corollary coincides with 
$\{a \in A_1(p) \colon \overline{a} \in V_p (\mathcal{L}_2 \cdots \mathcal{L}_\ell) - V_p [(1/p) E]\}$.
\end{proof}

\begin{proof}[Proof of Theorem \ref{thm:deterministic-root-counting01}$(2)$] First note that the decomposition \eqref{eqn:ascending-separable-factorization} stated under Definition \ref{hypo:ascending-seperable-factorization01} can be found via any classical factoring algorithm (see, e.g., \cite{BS96,BCS97}). 
The $\mathrm{gcd}$ of polynomials in $(\mathbb{Z}/p\mathbb{Z})[x]$ of degree $\leq d$ can be computed in near linear time $O(d^{1+o(1)} (\log p)^{1+o(1)})$, per an algorithm of Knuth and Sch{\"o}nhage \cite[Ch.~3]{BCS97}. Also, division with remainder for polynomials of degree $\leq d$ in $(\mathbb{Z}/p\mathbb{Z})[x]$ takes time $O(d^{1+o(1)} \log p)$, and reduction mod $p$ of a polynomial $f \in \mathbb{Z}[x]$ can be done in time linear in $\mathrm{size}(f) + \log{p}$ (see, e.g., \cite[Ch.~3]{BCS97} and \cite[Ch.~7]{BS96}). Finally, note that the $\mathrm{gcd}$ of $h_1$ and $x^p-x$ can be computed in time $O(d^{1+o(1)}(\log{p})^{1 +o(1)})$ by applying the binary method to the computation of $x^p \mod h_1$ (see, e.g., \cite[pp. 102-104, 121-122, \& 170-171]{BS96}). 

Going forward, we may assume that the maximal multiplicity $\ell \ge 1$. Now observe that $s_1:=\gcd{(h_1,x^{p} -x)} \in (\mathbb{Z}/p\mathbb{Z})[x]$ has the property that $V_p(h_1) = V_p(s_1)$ and $s_1$ has exactly $\mathrm{deg}(s_1)$ distinct linear factors. 
In particular, $s_1$ factors as $f_1f_2 \cdots f_\ell$. Next, note that $s_2:= h_1 /s_1$   factors as $g \cdot \prod_{i=1}^\ell f_i^{i -1}$. So then, $s_3:=\mathrm{gcd}(s_1,s_2) =  \prod_{i=2}^\ell f_i$. 
So we can then compute $f_1$ as $s_1/s_3$ and $h_2$ as $\mathrm{gcd}(s_3,t)$ within $(\mathbb{Z}/p\mathbb{Z})[x]$. This amounts to $3$ $\mathrm{gcds}$ and $2$ divisions in $(\mathbb{Z}/p\mathbb{Z})[x]$, which is clearly within the stated complexity bound - provided we can compute $t$ efficiently. That $t$ can be computed efficiently is immediate since it only involves a distinct degree factorization in $(\Z/p\Z)[x]$, a subtraction in $\Z[x]$, 
and a single polynomial division (by $p$) in $\mathbb{Z}[x]$.
\end{proof}

To conclude, now that the main arguments have been recorded,  we certainly invite readers to either: (a) generate many simple examples to better appreciate the root counting formula under Theorem 
\ref{thm:deterministic-root-counting01}(1); or (b) try implementing the algorithm in a computer algebra system they find palatable. We close the paper by providing the following example. 

\begin{example}
Fix the prime $p=5$, and consider the polynomial $f \in \mathbb{Z}[x]$ defined by 
\begin{align*}
f(x) &= x(x+2)^2(x+4)^{5}(x+3)^{14} (x^3 +2x+1)  + 5 (x+2)(x+4) \\ 
&= x^{25} + 66 x^{24} + 2073 x^{23} + 41225 x^{22} + 582597 x^{21} + 6225421 x^{20}+ 52256469 x^{19}\\
&+ 353428921 x^{18} + 1960388179 x^{17} + 9032286149 x^{16}+ 34894415443 x^{15} \\
&+ 113842103703 x^{14} + 315375403239 x^{13}+ 745101000855 x^{12} + 1506289490631 x^{11}\\
&+ 2610867590739 x^{10}+ 3879338706288 x^9 + 4921047219861 x^8 + 5275209809592 x^7\\
&+ 4688604525204 x^6 + 3350344836816 x^5 + 1835957176704 x^4\\
&+ 716433486336 x^3 + 174686782469 x^2 + 19591041054 x + 40.
\end{align*}
In particular, invoking language in the proof of Theorem \ref{thm:deterministic-root-counting01}(1), we have 
\begin{align*}
h_1 (x) &= x(x-3)^2(x-1)^5 (x-2)^{14}(x^3 +2x+1) \in  (\Z/p\Z) [x] \\ 
f_1 &=x , \quad f_2=(x-3) , \quad f_5=(x-1), \quad 
f_{14}=(x-2), \quad g=x^3 + 2x + 1 , \\ 
t(x) &= (x-3)(x-1), \quad h_2 = \gcd(f_2 f_5 f_{14},t ) = (x-3)(x-1) \in (\Z / 5 \Z) [x]. 
\end{align*}
Thus Theorem \ref{thm:deterministic-root-counting01}(1) says that  $$\# \{a \in A_2 (5) \colon f(a) \equiv 0 \mod 25 \} = \deg(f_1) + 5 \cdot \deg(h_2) = 1 + 5(2) = 11.$$ Now, $f(x) \equiv  0 \mod 5$ when $x = 0, 1, 2, 3 \in A_1 (5)$. 
The simple root $x= 0$ mod 5 lifts uniquely to the root $x=15$ mod 25 per Hensel's Lemma \ref{lem:hensel01}. Among the three degenerate roots mod 5,  only $x= 1$ and $x=3$ satisfy $f(x) \equiv 0 \mod 25$, and Hensel's Lemma \ref{lem:hensel02} lifts them 5-to-1. 
The values $x \in A_2 (5)$ for which $f(x) \equiv 0 \mod 25$ are 1, 3, 6, 8, 11, 13, 15, 16, 18, 21, and 23. We note that  
$1 \equiv 6 \equiv 11 \equiv 16 \equiv 21 \mod 5$, while  
$3 \equiv 8 \equiv 13 \equiv 18 \equiv 23 \mod 5$, as indicated under our discussion of Hensel's Lemma \ref{lem:hensel02}. In line with formula \ref{eqn:mod-p-roots-without-lifts} under Corollary \ref{cor:mod-p-roots-without-lifts}, we note in passing that only $x=2$ fails to lift to a root modulo 25. 
\end{example}

\section{Acknowledgements}
This research was conducted during  the MSRI Undergraduate Program in Summer 2017 with Dr. Federico Ardila serving as the on-site program director; we thank the Mathematical Sciences Research Institute 
and Dr. Ardila for the opportunity. 
The first three authors would like to thank their research advisor, Dr. J. Maurice Rojas, for exceptional guidance throughout the REU program. We thank two anonymous referees for feedback that improved the quality of exposition. Indeed, we thank one referee for asking whether we could deduce Corollary \ref{cor:mod-p-roots-without-lifts} above. We thank the Alfred P. Sloan Foundation and the National Science Foundation (Grant No. DMS-1659138) for providing financial support to run MSRI-UP this summer. 
We also acknowledge the partial REU support of MSRI's NSF grant (DMS-1440140),  NSF grant CCF-1409020,  
 and NSF DMS-1460766. 
Robert Walker was supported by a NSF GRF under Grant Number PGF-031543,  NSF RTG grant 0943832, and a Ford Foundation Dissertation Fellowship.

\bibliography{biblio} 

\begin{thebibliography}{10}

\bibitem{BS96}
E.~Bach and J.~Shallit.
\newblock {\em {Algorithmic Number Theory, Vol. I: Efficient Algorithms}}.
\newblock MIT Press, Cambridge, MA, 1996.

\bibitem{BCS97}
P.~Burgisser, M.~Clausen, and M.A. Shokrollahi.
\newblock {\em {Algebraic complexity theory, with the collaboration of Thomas
  Lickteig, Grundlehren der Mathematischen Wissenschaften, 315}}.
\newblock Springer-Verlag, Berlin, Cambridge, MA, 1997.

\bibitem{DummitFoote04}
D.S. Dummit and R.M. Foote.
\newblock {\em {Abstract Algebra, 3rd Edition}}.
\newblock Wiley Publishing, Hoboken, NJ, 2004.

\bibitem{10.2307/1970486}
H.~Hironaka.
\newblock Resolution of singularities of an algebraic variety over a field of
  characteristic zero: I.
\newblock {\em Annals of Mathematics}, 79(1):109--203, 1964.

\bibitem{Igusa78}
{J-I. Igusa}.
\newblock {\em {Forms of higher degree, Tata Institute of Fundamental Research,
  Bombay}}.
\newblock Narosa Publishing House, New Delhi, 1978.

\bibitem{Igusa00}
{J-I. Igusa}.
\newblock {\em {An Introduction to the Theory of Local Zeta Functions, AMS/IP
  Studies in Advanced Mathematics}}.
\newblock AMS, Providence, RI, 2000.

\bibitem{montgomery1991introduction}
H.~Montgomery, I.~Niven, and H.~Zuckerman.
\newblock {\em An Introduction to the Theory of Numbers, 5th Edition}.
\newblock John Wiley and Sons, Inc. New York, 1991.

\bibitem{weil1948courbes}
A.~Weil.
\newblock {\em Sur les courbes alg{\'e}briques et les vari{\'e}t{\'e}s qui s'
  en d{\'e}duisent}.
\newblock Number 1041. Hermann, 1948.

\bibitem{Weil49}
A.~Weil.
\newblock {Numbers of solutions of equations in finite fields}.
\newblock {\em Bull. Amer. Math. Soc.}, 55, 1949.
\newblock pp. 497--508.

\bibitem{zuniga2003computing}
W.A. Zuniga-Galindo.
\newblock {Computing Igusa's local zeta functions of univariate polynomials,
  and linear feedback shift registers}.
\newblock {\em Journal of Integer Sequences}, 6(2):3, 2003.

\end{thebibliography}

\end{document}